\documentclass{amsart}

\usepackage{amsfonts}
\usepackage{amsmath,amscd,amssymb,amsthm}
\usepackage{mathrsfs}
\usepackage{enumerate}
\usepackage{verbatim}
\usepackage{color}
\usepackage{tikz}
\usepackage{tikz-cd}
\usepackage{url}

\newtheorem{theorem}{Theorem}
\theoremstyle{plain}
\newtheorem{lemma}[theorem]{Lemma}
\newtheorem{proposition}[theorem]{Proposition}
\newtheorem{corollary}[theorem]{Corollary}

\theoremstyle{definition}
\newtheorem{definition}[theorem]{Definition}
\theoremstyle{remark}
\newtheorem{remark}[theorem]{Remark}

\newcommand{\R}{\mathbb{R}}
\newcommand{\Z}{\mathbb{Z}}
\newcommand{\F}{\mathbb{F}}

\begin{document}

\title{Bott Periodicity in the Hit Problem}

\author{Shaun V. Ault}

\address{Department of Mathematics and Computer Science\\ 
         Valdosta State University \\ 
         Valdosta, Georgia, 31698, USA.}

\email{svault@valdosta.edu}



\begin{abstract}
  In this short note, we use Robert Bruner's
  $\mathcal{A}(1)$-resolution of $P = \F_2[t]$ to shed light on the
  Hit Problem.  In particular, the reduced syzygies $P_n$ of $P$ occur
  as direct summands of $\widetilde{P}^{\otimes n}$, where
  $\widetilde{P}$ is the augmentation ideal of the map $P \to \F_2$.
  The complement of $P_n$ in $\widetilde{P}^{\otimes n}$ is free, and
  the modules $P_n$ exhibit a type of ``Bott Periodicity'' of period
  $4$: $P_{n+4} = \Sigma^8P_n$.  These facts taken together allow one
  to analyze the module of indecomposables in $\widetilde{P}^{\otimes
    n}$, that is, to say something about the ``$\mathcal{A}(1)$-hit
  Problem.''  Our study is essentially in two parts: First, we expound
  on the approach to the Hit Problem begun by William Singer, in which
  we compare images of Steenrod Squares to certain kernels of Squares.
  Using this approach, the author discovered a nontrivial element in
  bidegree $(5, 9)$ that is neither $\mathcal{A}(1)$-hit nor in
  $\mathrm{ker} Sq^1 + \mathrm{ker} Sq^3$.  Such an element is
  extremely rare, but Bruner's result shows clearly why these elements
  exist and detects them in full generality.  Second, we describe the
  graded $\F_2$-space of $\mathcal{A}(1)$-hit elements of
  $\widetilde{P}^{\otimes n}$ by determining its Hilbert series.
\end{abstract}

\maketitle

\section{Introduction and Definitions}\label{sec.intro}

The classical ``Hit Problem'' is concerned with finding a minimal
basis of
\[
  \F_2[x_1, x_2, \ldots, x_n] = P^{\otimes n} = H^*((\R
  P^{\infty})^{\times n}; \F_2)
\]
as a left $\mathcal{A}$-module, where $\mathcal{A}$ is the Steenrod
algebra at the prime $2$.  Indeed, we could ask about a ``Hit
Problem'' in any left $\mathcal{A}$-module $M$.
\begin{definition}
  Define $\mathcal{I}_M = \mathcal{A}^+M$.  
\end{definition}
We are primarily concerned with finding an $\F_2$-basis of the space
of indecomposables, or {\it unhit elements}, $M/\mathcal{I}_M$.  Since
$\mathcal{A}$ is generated as an algebra by the $2$-power squares,
$\{Sq^{2^i}\}_{i \geq 0}$, we may write:
\[
  \mathcal{I}_M = \mathrm{im} Sq^1 + \mathrm{im} Sq^2 + \mathrm{im}
  Sq^4 + \mathrm{im} Sq^8 + \cdots
\]
Consider the sub-Hopf algebra $\mathcal{A}(k)$ of $\mathcal{A}$
generated by $\{Sq^{2^i}\}_{i \leq k}$, and define:
\begin{definition}\label{def.I}
  For each integer $k \geq 0$,
  \[
    \mathcal{I}_M(k) = \mathcal{A}(k)^+M = \sum_{i=0}^{k} \mathrm{im}
    \left(Sq^{2^i} : M^{*-2^i} \to M^*\right).
  \]
\end{definition}
There is a filtration,
\[
  \mathcal{I}_M(0) \subseteq \mathcal{I}_M(1) \subseteq
  \mathcal{I}_M(2) \subseteq \mathcal{I}_M(3) \subseteq \cdots
  \subseteq \mathcal{I}_M.
\]
If an element $x \in M$ is hit, then it is certainly hit by $\theta
\in \mathcal{A}(k)$ for some $k \geq 0$.  Thus, the Hit Problem can be
studied by analyzing $M/\mathcal{I}_M(k)$, which is a more tractable
problem.  There is a strong relationship between $\mathcal{I}_M(k)$
and the kernels of ``spike squares,'' as defined below.
\begin{definition}\label{def.D}
  For each integer $k \geq 0$:
  \[
    \mathcal{D}_M(k) = \sum_{i=0}^{k} \mathrm{ker}
    \left(Sq^{2^{i+1}-1} : M^{*} \to M^{*+2^{i+i} - 1}\right).
  \]
\end{definition}
We call whole numbers that are one less than a power of 2 {\it
  spikes}, so the squares $Sq^{2^{i+1}-1}$ are often referred to as
{\it spike squares}.  Note, there is a similar filtration,
\[
  \mathcal{D}_M(0) \subseteq \mathcal{D}_M(1) \subseteq
  \mathcal{D}_M(2) \subseteq \mathcal{D}_M(3) \subseteq \cdots
\]
\begin{proposition}\label{prp.inclusion}
  For each $k \geq 0$,
  \begin{equation}\label{eqn.inclusion}
    \mathcal{I}_M(k) \subseteq \mathcal{D}_M(k).
  \end{equation}
\end{proposition}
\begin{proof}
  The inclusion~(\ref{eqn.inclusion}) is an easy consequence of the
  relation $Sq^{2m-1}Sq^m = 0$ in $\mathcal{A}$.
\end{proof}
What is surprising is that when $M = \F_2[x_1, x_2, \ldots, x_n]$, the
inclusion~(\ref{eqn.inclusion}) is almost an equality, in a sense that
will be made precise below.  Thus, the study of kernels of spike
squares sheds light on the Hit Problem.
\begin{remark}
  The author learned about this relationship between images and
  kernels from William Singer, and significant progress was made along
  these lines by Singer and the author~\cite{AS, A2}.  Singer's
  approach to the Hit Problem, by comparing images and kernels of
  Steenrod squares, to our knowledge, has not been considered
  elsewhere.
\end{remark}

Clearly, the inclusion $\mathcal{I}_M(k) \subseteq \mathcal{D}_M(k)
\subseteq M$ is as graded $\F_2$-vector spaces.  We define the
quotient:
\begin{definition}\label{def.U}
  \[
    \mathcal{U}_M(k) = \mathcal{D}_M(k) / \mathcal{I}_M(k).
  \]
\end{definition}
Then $\mathcal{U}_M(k)$ measures how far~(\ref{eqn.inclusion})
deviates from an equality.  This paper is primarily concerned with the
structure of $\mathcal{U}_M(1)$ when $M = \widetilde{P}^{\otimes n}$,
where $\widetilde{P}$ is the augmentation ideal of the map $P \to
\F_2$, so that $\widetilde{P}^{\otimes n} \cong \widetilde{H}^*((\R
P^{\infty})^{\wedge n}; \F_2)$.  We define the bigraded space,
\begin{definition}
  \[
    \mathbf{P} = \{ \mathbf{P}^{n,d} \}_{n \geq 1, d \geq 0}, \quad
    \textrm{where} \quad \mathbf{P}^{n,d} = \widetilde{H}^d((\R
    P^{\infty})^{\wedge n}; \F_2).
  \]
  Note, for all $k \geq 0$, the spaces $\mathcal{I}_M(k)$,
  $\mathcal{D}_M(k)$, and $\mathcal{U}_M(k)$ inherit this bigrading,
  and by abuse of notation, we identify these with their associated
  bigraded spaces:
  \begin{eqnarray*}
    \mathcal{I}_M(k) &=& \{ \mathcal{I}_M(k)^{n,d} \}_{n \geq 1, d
      \geq 0}\\ \mathcal{D}_M(k) &=& \{ \mathcal{D}_M(k)^{n,d} \}_{n
      \geq 1, d \geq 0}\\ \mathcal{U}_M(k) &=& \{
    \mathcal{U}_M(k)^{n,d} \}_{n \geq 1, d \geq 0}
  \end{eqnarray*}
\end{definition}
It is easy to see that $\mathcal{U}_{\mathbf{P}}(0) = 0$~\cite{AS},
but the structure of $\mathcal{U}_{\mathbf{P}}(1)$ is more delicate.
First, we observe that if the cohomological degree of $x \in M$ is
less than $4$, then $x$ may be in the kernel of $Sq^3$ with no chance
of being in the image of $Sq^2$.  For this reason, we call elements of
$\mathcal{U}_{\mathbf{P}}(k)$ in cohomological degrees $d < 2^{k+1}$
{\it degenerate elements}.  In preliminary work, we found that
$\mathcal{U}_{\mathbf{P}}(1)$ was devoid of non-degenerate elements in
all low bidgrees that could easily be checked by hand or computer.
However, further analysis and computation finally located a nontrivial
element of $\mathcal{U}_{\mathbf{P}}(1)$ in bidegree $(5,
9)$~\cite{A2}.  (Additional computer calculations were not made in
higher bidegrees after this discovery.)
\begin{remark}
  The author was working on the dual problem, as explained in
  Section~\ref{sec.dual}, and it is straightforward to show that the
  dual object $U_{\widetilde{\Gamma}}(k)$ is isomorphic to
  $\mathcal{U}_{\mathbf{P}}(k)$ as bigraded space.
\end{remark}

The reason that a nontrivial element exists in
$\mathcal{U}_{\mathbf{P}}(1)^{5,9}$ and that such nontrivial elements
are not found in lower bidegrees remained a mystery.  In a recent
email communication, Geoffrey Powell suggested that I look into a
preprint of Robert Bruner~\cite{Bruner}, which details the
$\mathcal{A}(1)$-module structure of $P^{\otimes n}$.  Bruner showed
that the there is an $\mathcal{A}(1)$-isomorphism:
\begin{equation}\label{eqn.Bruner_decomp}
  \widetilde{P}^{\otimes n} \cong P_n \oplus F_n,
\end{equation}
where $P_n$ is indecomposable and $F_n$ is a free
$\mathcal{A}(1)$-module.  Moreover, the modules $P_n$ exhibit a type
of ``Bott Periodicity'' of period 4:
\[
  P_{n+4} = \Sigma^8 P_n.
\]
Due to the periodic nature of $P_n$, the analysis of
$\mathcal{U}_{\mathbf{P}}(1)$ can be carried out in full.  This is the
main result of the paper, which is summarized in
Theorem~\ref{thm.hit_problem_A(1)}.  However, we would first like to
motivate the study of $\mathcal{U}_M(k)$ in general by recalling our
work on the dual ``$\mathcal{A}^+$-Annihilated Problem''.

I am grateful to Geoffrey Powell, Bob Bruner, and William Singer for
helpful conversations and advice.

\section{The Dual Problem}\label{sec.dual}

Suppose an $\mathcal{A}$-module $M$ is of finite type as graded $\F_2$
space.  Then its graded vector space dual $\#M$ is a right
$\mathcal{A}$-module which is degree-wise isomorphic to $M$.
\begin{remark}
  We avoid using the notation $M^*$ for dual of $M$, as the
  superscript position is better suited for degree.  In the
  literature, it seems that there are a number of competing notations
  for this basic concept, but I prefer the notation $\#M$ found on
  {\it nLab} (see
  \url{http://ncatlab.org/nlab/show/graded+vector+space}).  It is
  simple and does not get in the way of degree notations.
\end{remark}
We are particularly interested in the bigraded space,
$\widetilde{\Gamma} = \#\mathbf{P}_+$, whose components are given by:
\[
  \widetilde{\Gamma}_{n,d} = \left\{\begin{array}{ll}
  \widetilde{H}_d(S^0; \F_2), & n=0, \\ \widetilde{H}_d((\R
  P^{\infty})^{\wedge n}; \F_2), & n \geq 1.\end{array}\right.
\]
There is a unital asssociative connected bigraded algebra structure on
$\widetilde{\Gamma}$ induced by the natural mappings,
\[
  (\R P^{\infty})^{\wedge n_1} \times (\R P^{\infty})^{\wedge n_2} \to
  (\R P^{\infty})^{\wedge (n_1+n_2)},
\]
and unit map $\F_2 \stackrel{\cong}{\to} \widetilde{\Gamma}_{0,0}$.
This algebra structure provides an advantage to working in
$\widetilde{\Gamma}$ rather than in $\mathbf{P}$.  Indeed, David Anick
showed that $\widetilde{\Gamma}$ is a free (tensor) algebra~\cite{A1}.
So it makes sense to consider the dual question: {\it Find a basis of
  the space:}
\[
  \widetilde{\Gamma}_{\mathcal{A}^+} = \{ x \in \widetilde{\Gamma}
  \;|\; (x)Sq^n = 0\; \textrm{for all $n > 0$}\}.
\]
The correct analogs of
Defitions~\ref{def.I},~\ref{def.D}, and~\ref{def.U}
are obtained by dualizing the third column of the commutative
diagram below:

\begin{center}
  \begin{equation}\label{dia.hit_problem}
    \begin{tikzcd}
      & \mathcal{U}_M(k) & 0 & 0 & \\ 0 \rar & \mathcal{D}_M(k) \rar
      \uar & M \rar \uar & M/\mathcal{D}_M(k) \rar \uar & 0 \\ 0 \rar
      & \mathcal{I}_M(k) \rar \uar & M \rar \arrow[equals]{u} &
      M/\mathcal{I}_M(k) \rar \uar{\alpha} & 0 \\ & 0 \uar & 0 \uar &
      \mathrm{ker}(\alpha) \uar &\\
    \end{tikzcd}
  \end{equation}
\end{center}
Recall the definitions of~\cite{AS, A2}: For any right
$\mathcal{A}$-module $R$,
\begin{eqnarray*}
  I_R(k) &=& \bigcap_{i=0}^{k} \mathrm{im} \left(Sq^{2^{i+1}-1} :
  R_{*+2^{i+i} - 1} \to R_*\right)\\ \Delta_R(k) &=& \bigcap_{i = 0}^k
  \mathrm{ker}\left(Sq^{2^i} : R_* \to R_{*-2^i}\right)\\
  U_R(k) &=& \Delta_R(k)/I_R(k).
\end{eqnarray*}
Then upon dualizing the third column of~(\ref{dia.hit_problem}), we
obtain the exact sequence:
\begin{center}
\begin{tikzcd}
  0 \to I_{\#M}(k) \to \Delta_{\#M}(k) \to U_{\#M}(k) \to 0
\end{tikzcd}
\end{center}
Furthermore, there is a duality isomorphism,
\[
  U_{\#M(k)} \cong \mathrm{ker}(\alpha) \cong \mathcal{U}_M(k).
\]
For convenience, when considering $\widetilde{\Gamma}$, the notation
is abbreviated $I(k)$, $\Delta(k)$, and $U(k)$.  Each of these has a
bigrading inherited from $\widetilde{\Gamma}$.  The space of partially
$\mathcal{A}^+$-annihilateds, $\Delta(k)$, is in fact a free
subalgebra of $\widetilde{\Gamma}$ (see~\cite{AS}), and $I(k)$ is a
two-sided ideal of $\Delta(k)$, making $U(k)$ an algebra
(see~\cite{A2}, though the fact that $I(k)$ is a two-sided ideal was
known to Singer before the author included it in his publication).  In
our study of the $\mathcal{A}^+$-annihilated problem, we recognized
that finding elements of kernels ({\it i.e.,} $\Delta(k)$) may be a
more difficult task than finding elements of images ({\it i.e.,}
$I(k)$), and so we consider the following two problems:
\begin{enumerate}[I.]
  \item Identify the elements of $I(k)$ by giving a basis as
    space or algebra.
  \item Identify the elements of $U(k)$.
\end{enumerate}
In what follows, we return to the cohomological side of the Hit
Problem.

\section{Reduced $\mathcal{A}(k)$-modules}

We will be interested in left
$B$-modules, where $B$ is $\mathcal{A}$ or $\mathcal{A}(k)$.  Let
$\mathcal{M}^f$ be the category of $\Z$-graded left $B$-modules that
are trivial in sufficiently low degree, together with
degree-preserving $B$-maps.  Note, since we do not specialize to
unstable modules, the notions of {\it free} module and the loop
functor $\Omega$ are not the usual unstable versions.  Indeed, we have
$F$ free in $\mathcal{M}^f$ if and only $F \cong \bigoplus_{n_i\in
  \mathscr{I}} \Sigma^{n_i} B$, for some sequence $\mathscr{I}$ of
integers.  Let $I$ be the augmentation ideal of the counit $B \to
\F_2$.  We define the loop functor on $M \in \mathcal{M}^f$ by $\Omega
M = I \otimes M$.

We shall need the following result of Margolis (see Thm.~11.21 and
Prop.~13.13 of~\cite{Margolis}):
\begin{proposition}[Margolis]
  If $B$ is a bounded-below connected algebra of finite type, then
  every $B$-module $M$ has an expression $M \cong F \oplus
  M^{\mathrm{red}}$, where $F$ is free and $M^{\mathrm{red}}$ has no
  free summands.  Moreover, if $M \cong F' \oplus M'$ where $F'$ is
  free and $M'$ has no free summands, then $F' \cong F$ and $M' \cong
  M^{red}$.
\end{proposition}
\begin{remark}
  We call the module $M^{\mathrm{red}}$ the {\it reduced} part of $M$.
  The assignment $M \mapsto M^{\mathrm{red}}$ is not functorial.
\end{remark}
The next definition is of {\it stable isomorphism}, which is defined
for general modules elsewhere~\cite{Margolis, Bruner}, but for our
purposes can be defined more simply for bounded-below graded modules
as follows:
\begin{definition}
  Let $M$ and $N$ be modules in $\mathcal{M}^f$.  Then $M$ and $N$ are
  {\it stably isomorphic}, denoted $M \simeq N$, if $M^{\mathrm{red}}
  \cong N^{\mathrm{red}}$.
\end{definition}

We are particularly interested in $\mathcal{A}(1)$-modules.  Certainly
any $\mathcal{A}$-module is also an $\mathcal{A}(k)$-module for any
$k$ by restriction of scalars.
\begin{proposition}[Bruner~\cite{Bruner}, Cor.~3.3]
  For $n \geq 0$, there is a stable isomorphism of
  $\mathcal{A}(1)$-modules,
  \[
    \widetilde{P}^{\otimes (n+1)} \simeq
    \Omega^n\Sigma^{-n}\widetilde{P}.
  \]
\end{proposition}
\begin{definition}
  Define for each $n \in \Z$,
  \[
    P_{n+1} =
    \left(\Omega^n\Sigma^{-n}\widetilde{P}\right)^{\mathrm{red}}.
  \]
\end{definition}
The following result is a paraphrase of Bruner~\cite{Bruner},
Thm.~4.3.
\begin{proposition}\label{prp.P_n_structure}
  The modules $P_n$ are indecomposable and are determined by the
  following diagrams.  The arrows indicate $Sq^1$ and $Sq^2$.  An
  element's subscript identifies its degree.
  \begin{itemize}
    \item $P_1 = \widetilde{P}$:
      \[
        \begin{tikzcd}[column sep=small]
          t_1 \rar & t_2 \arrow[bend left]{rr} & t_3 \rar \arrow[bend
            left]{rr} & t_4 & t_5 \rar & t_6 \arrow[bend left]{rr} &
          t_7 \rar \arrow[dotted,bend left]{rr} & t_8 & {}
        \end{tikzcd}
      \]
    \item $P_2$:
      \[
        \begin{tikzcd}[column sep=small]
          & u_3 \rar \arrow[bend left]{rr} & u_4 \arrow{rrd} & u_5
          \rar & u_6 \arrow[bend left]{rr} & u_7 \rar \arrow[bend
            left]{rr} & u_8 & u_9 \rar[dotted] & {} \\ u'_2 \rar
          \arrow{rru} & u'_3 \arrow[bend left]{rr} & & u'_5 \rar &
          u'_6
        \end{tikzcd}
      \]
    \item $P_3$:
      \[
        \begin{tikzcd}[column sep=small]
          &&& v_6 \arrow[bend left]{rr} & v_7 \rar \arrow[bend
            left]{rr} & v_8 & v_9 \rar & v_{10} \arrow[bend left]{rr}
          & v_{11} \arrow[dotted,bend left]{rr} \rar & v_{12} & {}
          \\ v'_3 \rar \arrow[bend left]{rr} & v'_4 \arrow[bend
            right]{rr} & v'_5 \arrow{ru} \arrow[bend left]{rr} & v'_6
          \rar & v'_7
        \end{tikzcd}
      \]
    \item $P_4$:
      \[
        \begin{tikzcd}[column sep=small]
          w_7 \rar \arrow[bend left]{rr} & w_8 & w_9 \rar & w_{10}
          \arrow[bend left]{rr} & w_{11} \rar \arrow[bend left]{rr} &
          w_{12} & w_{13} \rar[dotted] & {}
        \end{tikzcd}
      \]
    \item $P_{n+4} = \Sigma^8P_n$, for $n \geq 1$.
  \end{itemize}
\end{proposition}

\section{Reduction of The Hit Problem to Reduced Modules}

The Hit Problem is concerned with identifying the elements of an
$\mathcal{A}$-module $M$ that fail to be in the image of any positive
square.  In other words, we are interested in the complement of the
module generated by images of positive squares.  Certainly every
element of an $\mathcal{A}$-module $M$ that lies outside of
$\mathcal{D}_{\mathbf{P}}(k)$ is not hit by $\mathcal{A}(k)$.  The
more interesting question is whether or not there are nontrivial
elements of $\mathcal{U}_{\mathbf{P}}(k)$.  In other words, to what
extent does the inclusion~(\ref{eqn.inclusion}) fail to be an
equality?  When $M$ is a free module, the question is easily
addressed.
\begin{lemma}\label{lem.free_modules}
  In a free $\mathcal{A}(k)$-module $F$, there is equality,
  $\mathcal{I}_F(k) = \mathcal{D}_F(k)$.  In other words,
  \[
    \mathcal{U}_F(k) = 0.
  \]
\end{lemma}
\begin{proof}
  Since $F$ is free, it suffices to consider a single free summand, so
  assume $F = \mathcal{A}(k)$.  Suppose $x \notin \mathcal{I}_F(k)$.
  Then $x = Sq^0 = 1$.  Since for each $0 \leq i \leq k$,
  $Sq^{2^{i+1}-1}x = Sq^{2^{i+1}-1} \in \mathcal{A}(k)$, it is clear
  that $x \notin \mathcal{D}_F(k)$.
\end{proof}
\begin{corollary}\label{cor.reduced}
  For an $\mathcal{A}(k)$-module $M$, there is an isomorphism,
  \[
    \mathcal{U}_{M}(k) \cong \mathcal{U}_{M^{red}}(k).
  \]
\end{corollary}

\section{The $\mathcal{A}(1)$-unhit elements of $\widetilde{P}^{\otimes n}$}

We now turn our attention to the classical Hit Problem, in which $M =
\widetilde{P}^{\otimes n}$.  By Cor.~\ref{cor.reduced}, we have
$\mathcal{U}_{\mathbf{P}}(1)^{n,*} \cong \mathcal{U}_{P_n}(1)$.
\begin{lemma}\label{lem.hit_problem_P_n}
  For $n \geq 1$, the modules $\mathcal{U}_{P_n}(1)$ are determined as
  follows:
  \begin{itemize}
    \item $\mathcal{U}_{P_1}(1) = \Sigma^1 \F_2$, on generator $t_1$.
    \item $\mathcal{U}_{P_2}(1) = \Sigma^2 \F_2$, on generator $u'_2$.
    \item $\mathcal{U}_{P_3}(1) = \{0\}$.
    \item $\mathcal{U}_{P_4}(1) = \{0\}$.
    \item $\mathcal{U}_{P_n+4}(1) = \Sigma^8\mathcal{U}_{P_n}(1)$, for
      $n \geq 1$.
  \end{itemize}    
\end{lemma}
\begin{proof}
  Easy exercise, based on the structure of $P_n$ indicated in
  Prop.~\ref{prp.P_n_structure}.
\end{proof}
\begin{theorem}\label{thm.hit_problem_A(1)}
  $\mathcal{U}_{\mathbf{P}}(1)$ is non-trivial only in bidegrees $(n,
  d) = (4m+r, 8m+r)$, for $m \geq 0$ and $r \in \{1, 2\}$, where
  $\mathcal{U}_{\mathbf{P}}(1)^{n,d} \cong \F_2$.  In other words, the
  generating function associated to the rank of
  $\mathcal{U}_{\mathbf{P}}(1)$ in each bidegree is:
  \[
    \sum_{n \geq 1, \,d \geq 0} \mathrm{dim}_{\F_2}(
    \mathcal{U}_{\mathbf{P}}(1)^{n,d} )s^nt^d
    = \frac{st(1 + st)}{1-s^4t^8}.
  \]
\end{theorem}
\begin{proof}
  Immediate from Lemma~\ref{lem.hit_problem_P_n}.
\end{proof}
\begin{remark}
  By duality, we expect the same structure in
  $U_{\widetilde{\Gamma}}(1)$.  In particular, the nontrivial element
  $z$ bidegree $(5,9)$ found in~\cite{A2} is evidently the dual of the
  image of $\Sigma^8t_1$ under the isomorphism of reduced modules,
  $\Sigma^8P_1 \stackrel{\cong}{\to} P_5 \subset
  \widetilde{P}^{\otimes 5}$.  In fact, according to Thm.~6.1 of
  Bruner~\cite{Bruner}, the nontrivial element of
  $\mathcal{U}_{\mathbf{P}}(1)^{5,9}$ is represented (in Bruner's
  notation) by:
  \[
    \Sigma^8t_1 \mapsto \langle 22221 \rangle + \langle
    \overline{1124}1 \rangle.
  \]
\end{remark}

\section{Ranks of spaces of the $\mathcal{A}(0)$-hit and $\mathcal{A}(1)$-hit elements}

In the section, we find generating functions for the ranks of
$\mathcal{I}_{\mathbf{P}}(0)$ and $\mathcal{I}_{\mathbf{P}}(1)$.  The
former is present in~\cite{A2} though not explicitly, while the latter
follows from Bruner's decomposition~(\ref{eqn.Bruner_decomp}).
\begin{definition}
  For $n \geq 1$ and $k \geq 0$, define the ordinary generating
  function:
  \[
    d_{k}(s,t) = \sum_{n \geq 1, d \geq 0} \mathrm{dim}_{\F_2}(
    \mathcal{I}_{\mathbf{P}}(k)^{n,d} )s^nt^d.
  \]
\end{definition}
\begin{remark}
  Of course, $d_{k}(t)$ is simply the bigraded {\it Hilbert series},
  $H\left(\mathcal{I}_{\mathbf{P}}(k)\right)$.
\end{remark}
\begin{proposition}
  \[
    d_{0}(s,t) = \frac{st^2}{(1+t)(1-t-st)}.
  \]
\end{proposition}
\begin{proof}
  Since for $n \geq 1$, $\mathcal{U}_{\mathbf{P}}(0)^{n,*} = 0$, there
  is a short exact sequence,
  \begin{center}
    \begin{equation}\label{eqn.ses_d_0}
      \begin{tikzcd}
        0 \rar & \mathcal{D}_{\mathbf{P}}(0)^{n,d} \rar
        \arrow[equals]{d} & \mathbf{P}^{n,d} \rar &
        \mathcal{I}_{\mathbf{P}}(0)^{n,d+1} \rar & 0 \\
        &
        \mathcal{I}_{\mathbf{P}}(0)^{n,d} 
      \end{tikzcd}
    \end{equation}
  \end{center}
  Note that the Hilbert Series of $\mathbf{P}$ is:
  \[
    H(\mathbf{P}) = \sum_{n \geq 1} \left(\frac{t}{1-t}\right)^ns^n
    = \frac{st}{1-t-st}.
  \]
  From the direct sum of sequences~(\ref{eqn.ses_d_0}), we obtain:
  \begin{eqnarray*}
    d_0(s,t) + t^{-1} d_0(s,t) &=& \frac{st}{1-t-st} \\ d_0(s,t) &=&
    \frac{st}{(1+t^{-1})(1-t-st)} \\&=& \frac{st^2}{(t+1)(1-t-st)}.
  \end{eqnarray*}
\end{proof}
\begin{remark}
  By duality, the correct generating function for $\Delta(0)$ is equal
  to:
  \[
    H\left(\mathbf{P}/ \mathcal{I}_{\mathbf{P}}(0)\right)=
    H\left(\mathbf{P}\right) - d_{0}(s,t) = \frac{st}{(1+t)(1-t-st)}.
  \]
  This generalizes the computations of Prop.~11 of~\cite{A2}.
\end{remark}

We now move on to the analysis of $k=1$.  
Suppose $A$ is any non-negatively-graded connected algebra over a
field $k$, and $M$ is an $A$-module.  If $F \to M \to 0$ is the start
of a {\it minimal} free resolution of $M$, then there is an
isomorphism of the modules of indecomposables,
\[
  M/A^+M \cong F/A^+F,
\]
as one can verify by a diagram-chase
of~(\ref{dia.indecomp}). Surjectivity of $\pi$ is clear.  We use the
minimality of $F$ to get injectivity.
\begin{center}
  \begin{equation}\label{dia.indecomp}
    \begin{tikzcd}
      0 \rar & A^+M \rar & M \rar & M/A^+M \rar & 0\\ 
      0 \rar & A^+F \uar \rar & F \rar \arrow[two heads]{u}
      & F/A^+F \rar \uar{\pi} & 0
    \end{tikzcd}
  \end{equation}
\end{center}
Thus, since $F \cong A \otimes (F/A^+F)$ as $k$-vector space,
\begin{equation}\label{eqn.Hilbert_series}
  H(M/A^+M) = H(F/A^+F) = \frac{H(F)}{H(A)}.
\end{equation}
Due to the complexity of the computations, it is useful to define the
singly-graded Hilbert series for $n \geq 1, k \geq 0$:
\[
  d_{n,k}(t) = H\left( \mathcal{I}_{\mathbf{P}}(k)^{n,*}\right).
\]
\begin{proposition}\label{prop.d_{n,1}(t)}
  For $n \geq 1$,
  \begin{equation}\label{eqn.d_{n,1}(t)}
    d_{n,1}(t) =
    \frac{t^{n+1}(1 + t^2 - t^3 - t^6) - t^{2n+3}(1-t)^nQ_{n+1}(t)}
         {(1-t)^n(1-t^4)(1+t^3)}.
  \end{equation}
  where
  \begin{equation}\label{eqn.Q_n(t)}
    Q_n(t) = \frac{1-t}{t^{2n}}H(P_n) = \left\{ \begin{array}{ll}
      t^{-1}, & \textrm{if $n \equiv 0, 1 \mod
        4$}\\ t^{-2}+t^{-1}-1+t-t^3, & \textrm{if $n \equiv 2 \mod
        4$}\\ t^{-3}+1-t^2, & \textrm{if $n \equiv 3 \mod 4$}\\
    \end{array}\right.
  \end{equation}
\end{proposition}
\begin{proof}
  Using~(\ref{eqn.Hilbert_series}), we obtain for any $k \geq 0$:
  \[
    H(\mathbf{P}^{n,*}/ \mathcal{I}_{\mathbf{P}}(k)^{n,*}) =
    \frac{H(F)}{H(\mathcal{A}(k))},
  \]
  where $F \to \widetilde{P}^{\otimes n} \to 0$ is the start of a
  minimal free $\mathcal{A}(k)$-resolution.  For $k=1$, Bruner
  identifies the maximal free summand of $\widetilde{P}^{\otimes n}
  \cong P_n \oplus F_n$ by its Hilbert series~\cite{Bruner}.  Set
  $f_n(t) = H(F_n)/H(\mathcal{A}(1))$.  Then,
  \begin{equation}\label{eqn.f_n(t)}
    f_n(t) =
    \frac{t^n\left(1-t^n(1-t)^{n-1}Q_n(t)\right)}{(1-t)^{n-1}(1-t^4)(1+t^3)},
  \end{equation}
  Bruner also provides a minimal free resolution for each of the
  $P_n$.  We are interested only in the start of each resolution,
  which we denote $F'_n \to P_n$.  (Note, in~\cite{Bruner}, the
  notation $F_n$ is again used for this purpose, though this is quite
  different than the $F_n$ used in the isomorphism
  $\widetilde{P}^{\otimes n} \cong P_n \oplus F_n$.)  To summarize,
  \begin{eqnarray*}
    F'_1 &\cong& \Sigma\mathcal{A}(1) \oplus \bigoplus_{i \geq 0}
    \Sigma^{4i+3} \mathcal{A}(1)\\ F'_2 &\cong& \Sigma^2\mathcal{A}(1)
    \oplus \bigoplus_{i \geq 0} \Sigma^{4i+3} \mathcal{A}(1)\\ F'_3
    &\cong& \bigoplus_{i \geq 0} \Sigma^{4i+3} \mathcal{A}(1)\\ F'_4
    &\cong& \bigoplus_{i \geq 0} \Sigma^{4i+7}
    \mathcal{A}(1)\\ F'_{4q+r} &\cong& \Sigma^{8q}F'_{r}
  \end{eqnarray*}    
  Set $g_n(t) = H(F'_n)/H(\mathcal{A}(1))$.  It is convenient to
  define $g_0(t) = t^{-8}g_4(t)$.  While the above descriptions of
  $F'_n$ can be used to produce Hilbert series for each, it will
  clarify the argument to relate $g_n(t)$ the polynomials $Q_n(t)$.
  (By the way, these are the same $Q_n(t)$ as in~\cite{Bruner}.)  In
  the proof of Theorem~4.3 of~\cite{Bruner}, four short exact seqences
  are presented connecting the syzygies $P_n$ and $P_{n+1}$ through
  $F'_n$, for $n = 1, 2, 3, 4$.
  \[
    0 \longrightarrow \Sigma P_{n+1} \longrightarrow F'_n
    \longrightarrow P_n \longrightarrow 0
  \]
  Applying $\Sigma^{-8}$ to the s.e.s.~corresponding to $n=4$, we
  obtain a similar sequence for $n=0$.  Thus, for each $n = 0, 1, 2,
  3$,
  \begin{eqnarray*}
    H(F'_n) &=& H(P_n) + tH(P_{n+1}) \\ &=& \frac{t^{2n}}{1-t} Q_n(t)
    + \frac{t^{2n + 3}}{1-t} Q_{n+1}(t)\\ &=& \frac{t^{2n}}{1-t}\left(
    Q_n(t) + t^3 Q_{n+1}(t) \right).\\
  \end{eqnarray*}
  Therefore, for $n = 4q + r$ with $0 \leq r < 4$,
  \begin{eqnarray*}
    g_n(t) &=& \frac{t^{8q}H(F'_r)}{H(\mathcal{A}(1))}\\ &=&
    \frac{t^{8q+2r}\left( Q_n(t) + t^3 Q_{n+1}(t)
      \right)}{(1-t)(1+t)(1+t^2)(1+t^3)}\\ &=& \frac{t^{2n}\left(
      Q_n(t) + t^3 Q_{n+1}(t) \right)}{(1-t^4)(1+t^3)}.
  \end{eqnarray*}

  Observe that $(F_n \oplus F'_n) \to \widetilde{P}^{\otimes n}$ is
  the start of a minimal free resolution, hence
  \begin{eqnarray*}
    H(\mathbf{P}^{n,*}/ \mathcal{I}_{\mathbf{P}}(1)^{n,*}) &=& f_n(t)
    + g_n(t)\\ &=& \frac{ t^n\left(1 -
      t^{n}(1-t)^{n-1}Q_n(t)\right)}{(1-t)^{n-1}(1-t^4)(1+t^3)} +
    \frac{t^{2n}\left(Q_n(t) + t^3
      Q_{n+1}(t)\right)}{(1-t^4)(1+t^3)}\\ &=& \frac{ t^n +
      t^{2n+3}(1-t)^{n-1}Q_{n+1}(t)}{(1-t)^{n-1}(1-t^4)(1+t^3)}.
  \end{eqnarray*}
  Formula~(\ref{eqn.d_{n,1}(t)}) results from the identification,
  \[
    H(\mathcal{I}_{\mathbf{P}}(1)^{n,*}) = H(\mathbf{P}^{n,*}) -
    H(\mathbf{P}^{n,*}/ \mathcal{I}_{\mathbf{P}}(1)^{n,*}).
  \]
  \begin{eqnarray*}
    d_{n,1}(t) &=& \frac{t^n}{(1-t)^n} - \frac{ t^n +
      t^{2n+3}(1-t)^{n-1}Q_{n+1}(t)}{(1-t)^{n-1}(1-t^4)(1+t^3)}\\ &=&
    \frac{t^{n+1}(1 + t^2 - t^3 - t^6) - t^{2n+3}(1-t)^nQ_{n+1}(t)}
         {(1-t)^n(1-t^4)(1+t^3)}.
  \end{eqnarray*}
\end{proof}
\begin{remark}
  By duality,
  \[
    H(\Delta(1)_{n,*}) = H(\mathbf{P}^{n,*}/
    \mathcal{I}_{\mathbf{P}}(1)^{n,*}) = \frac{ t^n +
      t^{2n+3}(1-t)^{n-1}Q_{n+1}(t)}{(1-t)^{n-1}(1-t^4)(1+t^3)}.
  \]
\end{remark}
Using the result of Prop.~\ref{prop.d_{n,1}(t)}, the bigraded Hilbert
series $d_1(s,t)$ can be found.
\[
  d_1(s,t) = \frac{st^2\left[R_1(s,t) -
      R_2(s,t)\right]}{(1-t^4)(1+t^3)},
\]
where
\begin{eqnarray*}
  R_1(s,t) &=& \frac{1+t^2-t^3-t^6}{1-t-st}\\ R_2(s,t) &=&
  \frac{t+t^2-t^3+t^4-t^6+s(t^2+t^5-t^7) + s^2t^6 + s^3t^8}{1-s^4t^8}.
\end{eqnarray*}

\section{Conclusion}

We believe, in principle, that similar methods could lead to a full
analysis of $\mathcal{U}_{\mathbf{P}}(k)$, for $k > 1$, as long as
nice decompositions of $\widetilde{P}^{\otimes n}$ into free
$\mathcal{A}(k)$-module and periodic complement can be constructed.
However the difficulty of such a task must be monumental, as
$\mathcal{A}(k)$ becomes much larger and harder to understand with
increasing $k$.  The $\F_2$-dimension of $\mathcal{A}(2)$ is $64$, so
creating such nice diagrams as Bruner has done for
$\mathcal{A}(1)$-resolutions~\cite{Bruner_diagrams} would be out of
reach.


\bibliographystyle{plain}


\end{document}